\documentclass[11 pt]{amsart}
\usepackage[all]{xy}

\usepackage{amssymb,amsmath,color}

\newtheorem{theorem}{Theorem}

\newtheorem{lemma}[theorem]{Lemma}

\newtheorem{conjecture}[theorem]{Conjecture}

\theoremstyle{remark}

\def\l{{\mathfrak l}}

  \let\l\lambda   
      
\let\GL\Lambda

\def\GL{\mathbf{GL}}
\def \GL2 {{\text{GL}_2}}

\def\Z{{\mathbb Z}}

\newcommand*\HYPERskip{&}
\catcode`,\active
\newcommand*\pFq{
\begingroup
\catcode`\,\active
\def ,{\HYPERskip}%
\doHyper
}
\catcode`\,12
\def\doHyper#1#2#3#4#5{%
\, _{#1}F_{#2}\left[\begin{matrix}#3 \smallskip \\  #4\end{matrix} \; ; \; #5\right]%
\endgroup
}

\title{On a conjecture of Kimoto and Wakayama}
\author{Ling Long, Robert Osburn and Holly Swisher}
\address{Department of Mathematics,  Louisiana State University, 303 Lockett Hall, Baton Rouge, LA 70803, USA}
\email{llong@math.lsu.edu}

\address{School of Mathematics and Statistics, University College Dublin, Belfield, Dublin 4, Ireland}
\email{robert.osburn@ucd.ie}

\address{Department of Mathematics, Oregon State University, 368 Kidder Hall, Corvallis, OR 97331, USA}
\email{swisherh@math.oregonstate.edu}


\begin{document}

\begin{abstract}
We prove a conjecture due to Kimoto and Wakayama from 2006 concerning Ap{\'e}ry-like numbers associated to a special value of a spectral zeta function. Our proof uses hypergeometric series and $p$-adic analysis.
\end{abstract}

\date{\today}

\maketitle

\section{Introduction}

Let $Q=Q_{\alpha, \beta}$ be the ordinary differential operator on $L^{2}(\mathbb{R}) \otimes \mathbb{C}^2$ defined by

\[
Q:= \begin{pmatrix} \alpha & 0 \\ 0 & \beta \end{pmatrix} \Biggl( -\frac{1}{2} \frac{d^2}{dx^2} + \frac{1}{2}x^2 \Biggr) + \begin{pmatrix} 0 & -1 \\ 1 & 0 \end{pmatrix} \Biggl( x \frac{d}{dx} + \frac{1}{2} \Biggr)
\]

\noindent where $\alpha$, $\beta$ are positive real numbers satisfying $\alpha \beta > 1$. The system defined by the operator $Q$ is called the {\it non-commutative harmonic oscillator} \cite{pw}. The operator $Q$ is positive, self-adjoint and unbounded with a discrete spectrum in which the multiplicities of the eigenvalues

\[
0 < \lambda_{1} \leq \lambda_{2} \leq \lambda_{3} \dotsc (\to \infty)
\]

\noindent are uniformly bounded. Thus, one can define the {\it spectral zeta function}

\[
\zeta_{Q}(s) : = \sum_{n=1}^{\infty} \frac{1}{\lambda_{n}^s}.
\]

The series $\zeta_{Q}(s)$ is absolutely convergent, defines a holomorphic function in $s$ for $\operatorname{Re} (s) > 1$ and can be meromorphically
continued to $\mathbb{C}$ (for details, see \cite{iw1}, \cite{iw2}). In \cite{KW}, Kimoto and Wakayama discuss the Ap{\'e}ry-like numbers

\begin{equation*}
\tilde{J}_{2}(n) := \sum_{k=0}^{n} (-1)^k \binom{-\frac12}{k}^2\binom{n}{k}
\end{equation*}

\noindent which occur in a representation of the special value $\zeta_{Q}(2)$. Similar to the Ap{\'e}ry numbers for $\zeta(2)$ and $\zeta(3)$, these numbers satisfy the recurrence relation (see Proposition 4.11 in \cite{iw2})

\[
4n^2 \tilde{J}_{2}(n) - (8n^2 - 8n+3) \tilde{J}_{2}(n-1) + 4(n-1)^2 \tilde{J}_{2}(n-2) = 0,
\]

\noindent with $\tilde{J}_{2}(0)=1$ and $\tilde{J}_{2}(1)=\frac{3}{4}$, possess many interesting arithmetic properties such as

\[
\tilde{J}_{2}(mp^r) \equiv \tilde{J}_{2}(mp^{r-1}) \pmod{p^r}
\]

\noindent for integers $m$, $r \geq 1$ and primes $p \geq 3$ (see Theorem 6.2 of \cite{KW}) and have the modular parametrization (see Theorem 5.1 in \cite{kw2} or $\# 19$ in Zagier's list \cite{z})

\[
\frac{\eta(2z)^{22}}{\eta(z)^{12} \eta(4z)^{8}} = \sum_{n=0}^{\infty} \tilde{J}_{2}(n) \, t^n
\]

\noindent where

\[
t=t(z)=16\frac{\eta(z)^8 \eta(4z)^{16}}{\eta^{24}(2z)}
\]

\noindent and $\eta(z)$ is the Dedekind eta-function. Our interest concerns the following conjecture from \cite{KW}.

\begin{conjecture}(Kimoto-Wakayama) \label{K-Wconjecture}
For primes $p\geq 3$,
\[
\sum_{k=0}^{p-1} \tilde{J}_{2}(k)^2 \equiv (-1)^{\frac{p-1}{2}} \pmod{p^3}.
\]
\end{conjecture}

In this paper, we prove two results, the second of which is equivalent to Conjecture \ref{K-Wconjecture}. Recall that for a nonnegative integer $r$ and $\alpha_i$, $\beta_i \in \mathbb{C}$ with $\beta_i\not\in\{\ldots, -3, -2, -1\}$, the (generalized) hypergeometric series $_{r+1}F_{r}$ is defined by
\[
\pFq{r+1}{r}{\alpha_1, \alpha_2, \ldots ,\alpha_{r+1}}{ ,\beta_1  , \ldots , \beta_r}{\l} :=
\sum_{k= 0}^{\infty} \frac{(\alpha_1)_k(\alpha_2)_k \ldots (\alpha_{r+1})_k}{(\beta_1)_k
\ldots (\beta_r)_k} \cdot \frac{\l^k}{k!},
\]
where $(a)_{0}:=1$ and $\displaystyle{(a)_k :=a(a+1)\cdots(a+k-1)}$.  This series converges for $|\l|<1$. Hypergeometric series are an important class of special functions which have been investigated by Gauss, Euler, and Kummer and have numerous applications to the theory of differential equations, algebraic varieties and physics. For a thorough treatment of hypergeometric series, the reader is referred to \cite{AAR}. Note that

\[
\tilde{J}_{2}(n) = \pFq{3}{2}{\frac12 & \frac12 & -n}{& 1 & 1}{1}.
\]

\begin{theorem}\label{K-W}
For primes $p>3$,

\begin{equation} \label{main1}
\sum_{x=0}^{p-1}\, \pFq{3}{2}{\frac {1-p}2,\frac {1+p}2,-x}{,1,1}{1}^2 \equiv (-1)^{\frac{p-1}2} \pmod{p^3}
\end{equation}

\noindent and for primes $p \geq 3$

\begin{equation} \label{main2}
\sum_{x=0}^{p-1}\, \pFq{3}{2}{\frac 12,\frac 12,-x}{,1,1}{1}^2\equiv (-1)^{\frac{p-1}2} \pmod{p^3}.
\end{equation}

\end{theorem}

The proof of Theorem \ref{K-W} uses hypergeometric series and $p$-adic analysis. The paper is organized as follows. In Section 2, we briefly recall the required background concerning hypergeometric series, then prove Theorem \ref{K-W}.  Finally, we have numerically observed the following generalization of (\ref{main2}): for primes $p \geq 3$ and integers $r \geq 1$,
\[
 \sum_{x=0}^{p^r-1} \pFq{3}{2}{\frac12 & \frac12 & -x}{&1&1}{1}^2 \equiv (-1)^{\frac{p-1}{2}} \sum_{x=0}^{p^{r-1}-1} \pFq{3}{2}{\frac12 & \frac12 & -x}{&1&1}{1}^2 \pmod {p^{3r}}.
\]
We leave this to the interested reader.

\section{Proof of Theorem \ref{K-W}}\label{proof}

The proof below is motivated by the approach of Rutkowski in \cite{Rut}.  We start with some  preliminaries.

\subsection{Preliminaries}

\begin{lemma}\label{lem:2}  Given integers $j,k,m$, with $m\geq 1$, and $j,k\geq 0$,
\[
\sum_{x=0}^{m-1} (x-j+1)_{j+k} = \frac{(m-j)_{j+k+1}}{j+k+1}.
\]
\end{lemma}

\begin{proof}
First, we observe that the identity holds trivially when $m \leq j$ since both sides are $0$.  Thus we assume $m>j$.  Moreover, the identity holds when $j=k=0$ as both sides are $m$. We note that $(x)_{n+1} - (x-1)_{n+1} = (n+1)(x)_n$ holds for integers $x\geq 0$, $n\geq 1$.  Then for any positive integer $N$,
\[
(N)_{n+1} = \sum_{x=0}^N \left(  (x)_{n+1} - (x-1)_{n+1} \right) = (n+1)\sum_{x=0}^N (x)_n.
\]
Letting $N=m-j$ and $n=j+k$ gives
\[
\sum_{x=0}^{m-j} (x)_{j+k} = \frac{(m-j)_{j+k+1}}{j+k+1},
\]
which is equivalent to
\[
\sum_{x=j}^{m-1} (x-j+1)_{j+k} = \frac{(m-j)_{j+k+1}}{j+k+1}.
\]
Since $(x-j+1)_{j+k}=0$ for $0\leq x<j$, this yields the lemma.
\end{proof}

We now fix some notation for the duration of the paper.   Since one can verify (\ref{main2}) directly for $p=3$, we fix $p > 3$ prime and $n=\frac{p-1}{2}$. Given a function $g(x)$, we define (see \cite{Rut})

\[
\displaystyle I(g)=\sum_{x=0}^{p-1} g(x).
\]

\noindent Let $f_n(x)$ be the degree $n$ polynomial in $\Z_p[x]$ defined by
\begin{equation}\label{fdefinition}
f_n(x)=\sum_{j=0}^n \binom{n}{j}\binom{n+j}{j}\binom{x}{j}=\pFq{3}{2}{-n,1+n,-x}{,1,1}{1}.
\end{equation}
These are orthogonal polynomials satisfying the following recursion (see (4) of \cite{vH})
\[
(n+1)^2f_{n+1}(x)=(2n+1)(2x+1)f_n(x)+n^2f_{n-1}(x).
\]
Furthermore, let $g(x)$ be the degree $p-1$ polynomial in $\Z_p[x]$ defined by
\begin{equation}\label{gdefinition}
g(x) = \sum_{j=0}^{p-1} (-1)^j \binom{-\frac12}{j}^2\binom{x}{j}=\pFq{3}{2}{\frac{1}{2},\frac{1}{2},-x}{,1,1}{1}_{p-1}
\end{equation}

\noindent where the subscript in (\ref{gdefinition}) denotes the truncation of the sum at $p-1$.

\subsection{Relationship between (\ref{main1}) and (\ref{main2})}
With our new notation, (\ref{main1}) is equivalent to 

\begin{equation} 
I(f_n(x)^2) \equiv (-1)^n \pmod{p^3}, \label{one}
\end{equation}

\noindent while (\ref{main2}) is equivalent to 

\begin{equation*}
I(g(x)^2) \equiv (-1)^n \pmod{p^3}. \label{two}
\end{equation*}
First, by (\ref{fdefinition}) and (\ref{gdefinition}), we observe that
\[
g(x) - f_n(x) = \sum_{k=1}^n \frac{\bigl((\frac12)_k^2 - (\frac{1-p}{2})_k(\frac{1+p}{2})_k\bigr)(-x)_k}{k!^3} + \sum_{k=\frac{p+1}{2}}^{p-1} \frac{(\frac12)_k^2(-x)_k}{k!^3}.
\]
Since $(\frac12)_k^2 \equiv (\frac{1-p}{2})_k(\frac{1+p}{2})_k \pmod{p^2}$, and $(\frac12)_k\equiv 0 \pmod{p}$ for $k\geq \frac{p+1}{2}$, we see that $g(x)-f_n(x) \in p^2x\Z_p[x]$ of degree $p-1$.  Thus, $g(x) = f_n(x) + p^2h(x)$, where $h(x)\in x\Z_p[x]$ has degree $p-1$.
This yields that
\[
I(g(x)^2) \equiv I(f_n(x)^2) + 2p^2I(f_n(x)h(x)) \pmod{p^3}.
\]
Note that if we prove
\begin{equation}\label{three}
I(f_n(x)g(x)) \equiv I(f_n(x)^2) \pmod{p^3},
\end{equation}
then we can conclude $I(f_n(x)h(x))\equiv 0 \pmod{p}$ and thus $I(g(x)^2) \equiv I(f_n(x)^2)\pmod{p^3}$.  Hence, in order to prove Theorem \ref{K-W}, it suffices to prove \eqref{one} and \eqref{three}.

\subsection{Proof of Theorem \ref{K-W}}
From \eqref{fdefinition} and \eqref{gdefinition}, we have

\begin{align}
I(f_n(x)^2) &= \sum_{j=0}^n \binom{n}{j} \binom{n+j}{j} I\left(f_n(x) \cdot \binom{x}{j} \right)  \label{I(f^2)},\\
I(f_n(x)g(x)) &= \sum_{j=0}^{p-1} (-1)^j \binom{-\frac12}{j}^2 I\left(f_n(x)\cdot \binom{x}{j}\right) \label{I(fg)}.
\end{align}

\noindent Note that $\binom{n}{j}$, $\binom{n+j}{j}$, and $j!$ do not introduce any factors of $p$;  $\binom{-\frac12}{j}$ has no factors of $p$ when $0\leq j \leq n$, but does contain a copy of $p$ when $n<j\leq p-1$.  We also observe that
\begin{equation}\label{modp^2}
(-1)^j\binom{-\frac12}{j}^2 \equiv \binom{n}{j}\binom{n+j}{j} \pmod{p^2},
\end{equation}
so \eqref{three} is true modulo $p^2$. For a finer analysis we study $I\left(f_n(x) \cdot \binom{x}{j} \right)$ modulo $p^3$.

\begin{lemma}\label{cor:4}For any $j\geq 0$, $m\ge 1$,
\[
I\left(f_m(x) \cdot \binom{x}{j} \right)=(-1)^m \sum_{k=0}^m \binom{m}{k}\binom{m+k}{k}\frac{(-1)^k (p-j)_{j+k+1}}{j!k!(j+k+1)}.
\]
\end{lemma}
\begin{proof}
We use the following identity (see page 142 of \cite{AAR}).  When $m$ is a positive integer and both sides converge,

\begin{equation}\label{eq:57}
\pFq{3}{2}{-m & a & b}{& d & e}{1}=\frac{(e-a)_m}{(e)_m}\,\pFq{3}{2}{-m & a & d-b}{& d & a+1-m-e}{1}.
\end{equation}

\noindent Letting $a=1+m$, $b=-x$, $d=e=1$ in (\ref{eq:57}) yields
\[
f_m(x)=(-1)^m \, \pFq{3}{2}{-m & 1+m & 1+x}{& 1 & 1}{1}=(-1)^m\sum_{k=0}^m\binom{m}{k}\binom{m+k}{k}\binom{-1-x}{k},
\]
and thus
\[
 I\left(f_m(x)\cdot \binom{x}{j}\right) = (-1)^m\sum_{k=0}^m\binom{m}{k}\binom{m+k}{k} I \left( \binom{x}{j}\binom{-1-x}{k} \right).
\]
Since $\binom{x}{j}\binom{-1-x}{k} = \frac{(-1)^k(x-j+1)_{j+k}}{j!k!}$, Lemma \ref{lem:2} yields the result.
\end{proof}

From Lemma \ref{cor:4}, we are now able to analyze $I\left(f_n(x) \cdot \binom{x}{j} \right)$ modulo $p^3$.   We will use the following identities from Rutkowski \cite{Rut}.  For $j=0,1\cdots, n-1$,
\begin{equation}\label{eq:Rut1}
\sum_{k=0}^n \frac{(-1)^k}{j+k+1}\binom{n+k}{k}\binom{n}{k}=0, 
\end{equation} 
and for $j=n$
\begin{equation}\label{eq:Rut2}
\sum_{k=0}^n \frac{(-1)^k}{n+k+1}\binom{n+k}{k}\binom{n}{k}=\frac{(-1)^n}{2n+1}\binom{2n}{n}^{-1}.
\end{equation} 
We note that these identities are direct consequences of the Pfaff-Saalsch\"utz formula  (Theorem 2.2.6 of \cite{AAR}), which says that for $n\in \mathbb N$,
\begin{equation}
\pFq{3}{2}{-n,a,b}{,c,1+a+b-c-n}{1}=\frac{(c-a)_n(c-b)_n}{(c)_n(c-a-b)_n}.
\end{equation}
Letting $a=n+1$, $b=j+1$, and $c=1$ yields 
\[\sum_{k=0}^n \frac{(-1)^k}{j+k+1}\binom{n+k}{k}\binom{n}{k}=\frac{1}{j+1}\cdot \pFq{3}{2}{-n,n+1,j+1}{,1,j+2}{1}=\frac1{j+1}\cdot \frac{(-n)_n(-j)_n}{(1)_n(-1-n-j)_n},  
\]
which gives \eqref{eq:Rut1} and \eqref{eq:Rut2}. 

Let $H_k=\sum_{j=1}^k \frac 1j$ denote the $k$th harmonic number where $H_0:=0$.  Note that for $0\leq k<p$, $H_k\in\Z_p$.  The following lemma is the key for proving Theorem \ref{K-W}.

\begin{lemma}\label{lem:5} Let $p>3$ be prime and $n=\frac{p-1}{2}$. For integers $0\leq j \leq p-1$,

\begin{equation} \label{allj}
\begin{aligned}
I\left(f_n(x)\cdot \binom{x}{j}\right) \equiv p (-1)^{n+j} & \sum_{k=0}^n \frac{(-1)^k}{j+k+1}\binom{n}{k}\binom{n+k}{k} \\ 
& + p^2 (-1)^{n+j} \sum_{k=0}^n\binom{n}{k}\binom{n+k}{k}\frac{(-1)^{k} [H_k-H_j]}{j+k+1}  \pmod{p^2}.
\end{aligned}
\end{equation}

\noindent Moreover, when $0\leq j<n$,

\begin{equation} \label{j<n}
I\left(f_n(x)\cdot  \binom{x}{j}\right) \equiv p^2 (-1)^{n+j} \sum _{k=0}^n \frac{\binom{n}{k}\binom{n+k}{k} (-1)^{k}[H_k-H_j]}{j+k+1} \pmod{p^3},
\end{equation}

\noindent and when $j=n$,

\begin{equation} \label{j=n}
I\left (f_n(x)\cdot  \binom{x}{n}\right )   \equiv (-1)^n\binom{2n}{n}^{-1} + p^2 \sum _{k=0}^n \frac{\binom{n}{k}\binom{n+k}{k}(-1)^{k} [H_k-H_n]}{n+k+1}\pmod{p^3}. \\
\end{equation}

\end{lemma}
\begin{proof}

We first observe that  when $0\leq k\leq n$ and $0\leq j \leq p-1$, we have

\begin{equation} \label{facp}
\begin{aligned}
\frac{(p-j)_{j+k+1}}{j!k!} & = \left( \frac{p}{j} - 1 \right)\left( \frac{p}{j-1} - 1 \right) \cdots \left( \frac{p}{1} - 1 \right) p \left( \frac{p}{1} + 1 \right) \cdots \left( \frac{p}{k} + 1 \right) \\
& \equiv p (-1)^j(1 + p[H_k-H_j]) \pmod{p^3}
\end{aligned}
\end{equation}

\noindent and thus (\ref{allj}) follows from Lemma \ref{cor:4} with $m=n$  since $j+k+1$ introduces at most one factor of $p$ in the denominator. Now, if $0 \leq j < n$, then $p \nmid j+k+1$ and so  Lemma \ref{cor:4}, \eqref{eq:Rut1}, and (\ref{facp}) imply (\ref{j<n}). We now note that  letting $j=k=n$ in \eqref{facp} gives

\begin{equation} \label{k=n}
\frac{(p-n)_{2n+1}}{n!^2} \equiv p (-1)^n \pmod{p^3}.
\end{equation}

\noindent Thus, after taking $m=n$ in Lemma \ref{cor:4},  applying (\ref{k=n}) to the $k=n$ term, then
applying (\ref{facp}) with $j=n$ to the $0\leq k<n$ terms and recombining, we have

\begin{equation*}
\begin{aligned}
I\left (f_n(x)\cdot  \binom{x}{n}\right ) & \equiv (-1)^n \sum_{k=0}^{n-1} \binom{n}{k}\binom{n+k}{k}\frac{(-1)^k (p-n)_{n+k+1}}{n!k!(n+k+1)} + (-1)^n \binom{2n}{n} \pmod{p^3} \\
& \equiv p \sum_{k=0}^{n} \binom{n}{k} \binom{n+k}{k} \frac{(-1)^k}{n+k+1} + p^2 \sum _{k=0}^n \frac{\binom{n}{k}\binom{n+k}{k}(-1)^{k} [H_k-H_n]}{n+k+1}\pmod{p^3}.
\end{aligned}
\end{equation*}

\noindent Using \eqref{eq:Rut2}, we arrive at (\ref{j=n}).
\end{proof}

Finally, we need two additional lemmas. The first is from \cite{Morley}.

\begin{lemma}\label{lem:Morley} Let $p>3$ be prime and $n=\frac{p-1}2$. We have
\[
\binom{-\frac 12}{n}^2\equiv (-1)^n\binom{2n}{n} \pmod{p^3}.
\]
\end{lemma}

\begin{lemma}\label{lem:7} Let $p>3$ be prime and $n=\frac{p-1}{2}$. We have
 \[
 p \sum_{j=n+1}^{p-1} \binom{-\frac12}{j}^2 \sum_{k=0}^n \frac{(-1)^k}{j+k+1}\binom{n}{k}\binom{n+k}{k} \equiv 0\pmod{p^3}.
 \]
\end{lemma} 

\begin{proof}  We note that if $i$ is a fixed integer such that $1\le i\le n$, then  since $-n \equiv n+1 \pmod{p}$,

\begin{equation} \label{s1}
\begin{aligned}
(n-i)!^2&=(n-(n-1))^2\cdots(n-j)^2\cdots (n-i)^2\\
& \equiv (p-1)^2\cdots (n+1+j)^2\cdots (1+n+i)^2\pmod p.
\end{aligned}
\end{equation}

\noindent Thus $(n-i)!^2(n+i)!^2\equiv 1\pmod p$ by Wilson's theorem. Also 
\[
\left ( \frac 12 \right )_{n+i}  =\left ( \frac 12 \right )_n \left(\frac{p}{2} \right ) \left (\frac{p}2+1 \right)\cdots \left (\frac{p}2+ i-1 \right )  \equiv  \frac{p}{2} \left ( \frac 12 \right )_n (i-1)! \pmod{p^2},
\]
and thus
\begin{equation} \label{s2}
\left ( \frac 12 \right )_{n+i} ^2 \equiv \frac{p^2}{4} \left ( \frac 12 \right )_n^2 (i-1)!^2 \pmod{p^3}.
\end{equation}

\noindent Similarly,
\begin{equation} \label{s3}
\left( \frac 12 \right )_n^2 = \left ( \frac 12 \right )_{n-i}^2{\left(\frac{p}2-1\right)^2\cdots \left(\frac{p}2-i\right)^2} \equiv i!^2 \left ( \frac 12 \right )_{n-i}^2 \pmod{p}.
\end{equation}
\noindent By \eqref{modp^2}, it suffices to prove
\begin{equation*}
p \sum_{j=n+1}^{p-1} \sum_{k=0}^n \frac{(\frac12)_j^2}{j!^2}\frac{(\frac12)_k^2}{k!^2}\cdot \frac 1{j+k+1}\equiv 0 \pmod{p^3}.
\end{equation*}
Since $j\ge n+1$, we have $p^2\mid (\frac12)_j^2$ and thus the summand is $0$ modulo $p^3$ when $j+k+1\neq p$.  Using (\ref{s1})--(\ref{s3}), and the fact (see \cite{Long}) that $\sum_{i=1}^{n} \frac1{i^2} \equiv 0 \pmod{p}$, we have 

\begin{equation*}
\begin{aligned}
p \sum_{j=n+1}^{p-1} \sum_{k=0}^n \frac{(\frac12)_j^2}{j!^2}\frac{(\frac12)_k^2}{k!^2}\cdot \frac 1{j+k+1} & \equiv \sum_{j=n+1}^{p-1} \frac{(\frac12)_j^2}{j!^2}\frac{(\frac12)_{p-1-j}^2}{(p-1-j)!^2} \pmod{p^3} \\
& = \sum_{i=1}^n\frac{(\frac12)_{n+i}^2(\frac12)_{n-i}^2}{(n+i)!^2(n-i)!^2} \\
& \equiv \frac{p^2}4 \left(\frac12 \right)_{n}^4 \sum_{i=1}^{n} \frac1{i^2} \pmod{p^3} \\
& \equiv 0 \pmod{p^3}.
\end{aligned}
\end{equation*}
\end{proof}  

\noindent We now have the tools to prove Theorem \ref{K-W}.

\begin{proof}[Proof of Theorem \ref{K-W}]
We first split \eqref{I(f^2)} into the cases $j<n$ and $j=n$, and apply (\ref{j<n}) and (\ref{j=n}) to obtain

\begin{equation} \label{split}
 I(f_n(x)^2) \equiv (-1)^n+p^2 \cdot (-1)^n \sum_{j,k=0}^n\frac{\binom{n}{j}\binom{n+j}{j} \binom{n}{k}\binom{n+k}{k}(-1)^{k+j}[H_k-H_j]}{j+k+1}\pmod{p^3}.\\
\end{equation}

\noindent As the sum on the right-hand side of (\ref{split}) is symmetric in $j$ and $k$, it equals $0$ and thus \eqref{one} follows. We now split \eqref{I(fg)} into the cases when $j<n$, $j=n$, and $j>n$ to obtain $I(f_n(x)g(x)) = A + B + C$ where
\begin{equation*}
A = \sum_{j=0}^{n-1} (-1)^j \binom{-\frac{1}{2}}{j}^2 I\left(f_n(x)\cdot \binom{x}{j} \right),
\end{equation*}
\begin{equation*}
B = (-1)^n \binom{-\frac12}{n} I\left(f_n(x)\cdot  \binom{x}{n} \right) \\
\end{equation*}
and
\begin{equation*}
C = \sum_{j=n+1}^{p-1} (-1)^j \binom{-\frac{1}{2}}{j}^2 I\left(f_n(x)\cdot  \binom{x}{j} \right) .
\end{equation*}
When $0\leq j <n$, we see from (\ref{j<n}) that $I\left(f_n(x)\cdot  \binom{x}{j} \right) \equiv 0 \pmod{p}$.  Thus by \eqref{modp^2} and Lemma \ref{lem:Morley} we have that
\[
A + B \equiv I(f_n(x)^2) \pmod{p^3}.
\]
Since $p^2$ divides $\binom{-\frac12}{j}^2$ when $j>n$, applying (\ref{allj}) to $C$ yields
\begin{multline*}
C  \equiv p\cdot (-1)^n \sum_{j=n+1}^{p-1}\binom{-\frac{1}{2}}{j}^2\sum_{k=0}^n \binom{n}{k}\binom{n+k}{k} \frac{(-1)^k}{j+k+1} \\
+ p^2\cdot (-1)^n\sum_{j=n+1}^{p-1}\binom{-\frac{1}{2}}{j}^2\sum_{k=0}^n \binom{n}{k}\binom{n+k}{k} \frac{(-1)^k[H_k - H_j]}{j+k+1} \pmod{p^4}.
\end{multline*}

We observe that the first summand vanishes modulo $p^3$ by Lemma \ref{lem:7} and the second summand vanishes modulo $p^3$ since $j+k+1$ introduces at most one factor of $p$ in the denominator.  This proves \eqref{three}. 
\end{proof}

\section*{Acknowledgements}

The first author is supported by the NSF grant DMS-1303292 and the third author thanks Tulane University for hosting her during this project.  The authors would like to thank Ravi Ramakrishna for the reference \cite{Rut} and helpful conversations, and also Kazufumi Kimoto and Masato Wakayama for their encouragement and interest.   Many claims were first verified by  the open source mathematics software \texttt{Sage}.

\end{document}